\documentclass[10pt]{amsart}

\usepackage{fullpage}
\usepackage{amssymb}
\usepackage{amsmath}
\usepackage{amsxtra}
\usepackage{amscd}
\usepackage{graphicx}
\usepackage{amsfonts}
\usepackage{pb-diagram}
\usepackage{float}
\usepackage{enumerate}
\usepackage{dsfont}
\usepackage{color}

\numberwithin{equation}{section}
\newtheorem{thm}{Theorem}[section]
\newtheorem{lem}[thm]{Lemma}
\newtheorem{cor}[thm]{Corollary}

\newtheorem{rem}[thm]{Remark}

\def\C{\mathbb{C}}
\def\Z{\mathbb{Z}}

\def\cR{\mathcal{R}}

\def\cH{\mathcal{H}}

\def\cT{\mathcal{T}}

\def\Y{{\rm Y}_{d,n}^{\mathrm{aff}}(q)}

\begin{document}

\title[The affine Yokonuma--Hecke algebra and the pro-$p$-Iwahori--Hecke algebra]
  {The affine Yokonuma--Hecke algebra\\ and the pro-$p$-Iwahori--Hecke algebra}

\author{Maria Chlouveraki}
\address{Laboratoire de Math\'ematiques de Versailles, 
Universit\'e de Versailles St-Quentin, 
B\^atiment Fermat, 45 avenue des \'Etats-Unis, 78035 Versailles cedex, France.}
\email{maria.chlouveraki@uvsq.fr}

\author{Vincent S\'echerre}
\address{Laboratoire de Math\'ematiques de Versailles, 
Universit\'e de Versailles St-Quentin, 
B\^atiment Fermat, 45 avenue des \'Etats-Unis, 78035 Versailles cedex, France.}
\email{vincent.secherre@uvsq.fr}


\thanks{We are grateful to Guillaume Pouchin for his remarks and always 
  helpful explanations.  We would also like to thank Lo\"ic Poulain d'Andecy for our fruitful 
discussions on the topics of this paper.}


\begin{abstract}
We prove that the affine Yokonuma--Hecke algebra defined by 
Chlouveraki and Poulain d'Andecy is a particular case of the 
pro-$p$-Iwahori--Hecke algebra defined by Vign\'eras.
\end{abstract}

\maketitle

\section{Introduction}\label{sec-intro}

A family of complex algebras $\Y$, called affine 
Yokonuma--Hecke algebras, has been defined and studied by Chlouveraki and 
Poulain d'Andecy in \cite{ChPo2}. 
The existence of these algebras has been first mentioned by Juyumaya and 
Lambropoulou in \cite{jula4}. 
These algebras, which generalise both affine Hecke algebras of type $A$ and 
Yokonuma--Hecke algebras \cite{yo}, are used to 
determine the representations of Yokonuma--Hecke algebras \cite{ChPo} and 
construct invariants 
for framed and classical knots in the solid torus \cite{ChPo2}.
Moreover, when $q^2$ is a power of a prime number $p$ and $d=q^2-1$, one can 
verify that the affine Yokonuma--Hecke 
algebra $\Y$ is isomorphic to the convolution algebra 
of complex valued and compactly supported functions on the group 
${\rm GL}_{n}(F)$, with $F$ a suitable $p$-adic field, that are bi-invariant 
under the pro-$p$-radical of an Iwahori subgroup (cf.~\cite{Vi1}). 
It is natural to ask whether there is a family of algebras that generalises, 
in a similar way, affine Hecke algebras in arbitrary type. 

In a recent series of papers \cite{Vi2,Vi3,Vi4}, Vign\'eras introduced 
and studied a large family of algebras, called pro-$p$-Iwahori--Hecke algebras. 
They generalise convolution algebras of compactly supported functions on a 
$p$-adic connected reductive group that are bi-invariant under the 
pro-$p$-radical of an Iwahori subgroup, which play 
an important role in the $p$-modular representation 
theory of $p$-adic reductive groups (see \cite{AHHV} for instance).

In this note, we show that the algebra $\Y$ 
of Chlouveraki and Poulain d'Andecy is a pro-$p$-Iwahori--Hecke algebra in the 
sense of Vign\'eras \cite{Vi2}. 

\section{The affine Yokonuma--Hecke algebra}\label{sec-def}

Let $d,n \in \Z_{>0}$. Let $q$ be an indeterminate or a non-zero complex number, and set $\cR:=\C[q,q^{-1}]$. 
We denote by $\Y$ the associative algebra over $\cR$ generated by elements
$$t_1,\ldots,t_n,g_1,\ldots,g_{n-1},X_1,X_1^{-1}$$
subject to the following defining relations:
\begin{equation}\label{def-aff1}
\begin{array}{lcrclcl}
\text{(br1)} & &g_ig_j & = & g_jg_i && \mbox{for all $i,j=1,\ldots,n-1$ such that $\vert i-j\vert > 1$,}\\[0.1em]
\text{(br2)} & &g_ig_{i+1}g_i & = & g_{i+1}g_ig_{i+1} && \mbox{for  all $i=1,\ldots,n-2$,}\\[0.1em]
\text{(fr1)} & &t_it_j & =  & t_jt_i &&  \mbox{for all $i,j=1,\ldots,n$,}\\[0.1em]
\text{(fr2)} & &g_it_j & = & t_{s_i(j)}g_i && \mbox{for all $i=1,\ldots,n-1$ and $j=1,\ldots,n$,}\\[0.1em]
\text{(fr3)} &  &t_j^d   & =  &  1 && \mbox{for all $j=1,\ldots,n$,}\\[0.1em]
\text{(aff1)} &  &X_1\,g_1X_1g_1 & =  & g_1X_1g_1\,X_1  &&\\[0.1em]
\text{(aff2)}&  &X_1g_i & = & g_iX_1 && \mbox{for all $i=2,\ldots,n-1$,}\\[0.1em]
\text{(aff3)} &  &X_1t_j & = & t_jX_1 && \mbox{for all $j=1,\ldots,n$.}\\[0.1em]
\text{(quad)} &  &g_i^2  & = & 1 + (q-q^{-1}) \, e_{i} \, g_i && \mbox{for  all $i=1,\ldots,n-1$,}
\end{array}
\end{equation}
where   $s_i$ denotes the transposition $(i,i+1)$ and
\begin{equation}\label{e_i}
e_i :=\frac{1}{d}\sum\limits_{k=0}^{d-1}t_i^k t_{i+1}^{d-k}\,
\end{equation}
for all $i=1,\ldots,n-1$.
The algebra $\Y$ was called in \cite{ChPo} the \emph{affine Yokonuma--Hecke algebra} and was studied in \cite{ChPo2}, together with its cyclotomic quotients. 
This algebra is isomorphic to the modular framisation of the affine Hecke algebra; see definition in \cite[Section 6]{jula5} and Remark 1 in \cite{ChPo}.   
For $d=1$, ${\rm Y}_{1,n}^{\mathrm{aff}}(q)$ is the standard affine Hecke algebra of type $A$.

\begin{rem}{\rm Relations (br1), (br2), (aff1) and (aff2) are the defining relations of the \emph{affine braid group} $B_n^{\mathrm{aff}}$. 
Adding relations (fr1), (fr2) and (aff3) yields the definition of the \emph{extended affine braid group} or \emph{framed affine braid group} $\Z^n \rtimes  B_n^{\mathrm{aff}}$, with the $t_j$'s being interpreted as the ``elementary framings" (framing $1$ on the $j$th strand). The quotient of $\Z^n \rtimes  B_n^{\mathrm{aff}}$ over the relations (fr3) is the \emph{modular framed affine braid group} $(\Z/d\Z)^n \rtimes B_n^{\mathrm{aff}}$ (the framing of each braid strand is regarded modulo~$d$). 
Thus, the affine Yokonuma--Hecke algebra $\Y$ can be obtained as the quotient of the group algebra $\cR[(\Z/d\Z)^n \rtimes B_n^{\mathrm{aff}}]$   over the quadratic relation (quad).}
\end{rem}

Note that, for all $i=1,\ldots,n-1$, the elements $e_i$ are idempotents, and we have $g_i e_i = e_i g_i$. Moreover, 
the elements $g_i$ are invertible, with
\begin{equation}\label{inverse}
g_i^{-1} = g_i - (q- q^{-1})\, e_i  \qquad \mbox{for all $i=1,\ldots,n-1$}.
\end{equation}

Now, for $i,l=1,\ldots,n$, we set
\begin{equation}\label{e_il}
e_{i,l} :=\frac{1}{d}\sum\limits_{k=0}^{d-1}t_i^k t_{l}^{d-k}\,
\end{equation}
The elements $e_{i,l}$ are idempotents.  We also have $e_{i,i}=1$, $e_{i,i+1}=e_i$ and $e_{i,l}=e_{l,i}$. Moreover, it is easy to check that
\begin{equation}\label{useful e_i}
t_j e_{i,l} = t_{s_{i,l}(j)}e_{i,l} = e_{i,l} t_{s_{i,l}(j)}=e_{i,l}t_j \ \qquad \mbox{for all $j=1,\ldots,n$},
\end{equation}
where   $s_{i,l}$ denotes the transposition $(i,l)$.

We define inductively elements $X_2,\ldots,X_n$ of $\Y$ by
\begin{equation}\label{rec-X}
X_{i+1}:=g_iX_ig_i\ \ \ \ \text{for $i=1,\ldots,n-1$.}
\end{equation}
We have that the elements $t_1,\ldots,t_n,X_1^{\pm 1},\ldots,X_n^{\pm 1}$ form a commutative family  \cite[Proposition 1]{ChPo}.
Moreover, in \cite[Lemma 1]{ChPo}, it is proved that,
 for any $i\in\{1,\ldots,n\}$, we have
\begin{equation}\label{g-X}
g_jX_i=X_ig_j \quad\text{and} \quad g_jX_i^{-1}=X_i^{-1}g_j \ \ \ \text{for $j=1,\ldots,n-1$ such that $j\neq i-1,i$.}
\end{equation}
Finally, using (\ref{def-aff1})(quad) and (\ref{inverse}), it is easy to check that, for any $i\in\{1,\ldots,n\}$, we have
\begin{equation}\label{noname1}
g_iX_i=X_{i+1}g_i-(q-q^{-1})e_iX_{i+1} \quad\text{and} \quad  g_i X_{i+1}=X_ig_i+(q-q^{-1}) e_iX_{i+1}\,,
\end{equation}
which in turn yields
\begin{equation}\label{noname2}
g_iX_i^{-1}=X_{i+1}^{-1}g_i+(q-q^{-1})e_iX_{i}^{-1}\quad\text{and} \quad  g_i X_{i+1}^{-1}=X_i^{-1}g_i-(q-q^{-1})e_i X_{i}^{-1} \,.
\end{equation}
We can easily prove by induction, on $a,b \in \Z_{>0}$, that the following equalities hold:
\begin{equation}\label{noname3}
g_iX_i^a=X_{i+1}^ag_i-(q-q^{-1})e_i\sum_{k=0}^{a-1}X_i^{k}X_{i+1}^{a-k} \quad \text{and} \quad g_iX_{i+1}^b=X_i^bg_i+(q-q^{-1})e_i\sum_{k=0}^{b-1}X_i^{k}X_{i+1}^{b-k}
\end{equation}
\begin{equation}\label{noname4}
g_iX_i^{-a}=X_{i+1}^{-a}g_i+(q-q^{-1})e_i\sum_{k=0}^{a-1}X_i^{-a+k}X_{i+1}^{-k} \quad \text{and} \quad g_iX_{i+1}^{-b}=X_i^{-b}g_i-(q-q^{-1})e_i\sum_{k=0}^{b-1}X_i^{-b+k}X_{i+1}^{-k}
\end{equation} 
Note also that
\begin{equation}
g_i X_iX_{i+1} = g_iX_ig_iX_ig_i =X_{i+1}X_i g_i = X_iX_{i+1}g_i \quad \text{and} \quad g_i X_i^{-1}X_{i+1}^{-1}  = X_i^{-1}X_{i+1}^{-1}g_i. 
\end{equation}
The above formulas 
yield it turn the following lemma:
\begin{lem}\label{lem-form}
We have the following identities satisfied in $\Y$ ($i=1,\ldots,n-1$):
\begin{equation}\label{formAK}
g_iX_i^aX_{i+1}^b=\left\{\begin{array}{ll} \!\!X_i^bX_{i+1}^ag_i-(q-q^{-1})e_i\sum\limits_{k=0}^{a-b-1}X_i^{b+k}X_{i+1}^{a-k} & \text{if $a\geq b$,}\\[1.4em]
\!\!X_i^bX_{i+1}^ag_i+(q-q^{-1})e_i\sum\limits_{k=0}^{b-a-1}X_i^{a+k}X_{i+1}^{b-k} & \text{if $a\leq b$,}
\end{array}\right.\ \qquad a,b\in\Z\,.
\end{equation}
\end{lem}

Let $w \in \mathfrak{S}_n$,  where $\mathfrak{S}_n$ is the symmetric group on $n$ letters, and let $w=s_{i_1}s_{i_2}\ldots s_{i_r}$ be a reduced expression for $w$. 
Since the generators $g_i$ of $\Y$ satisfy the same braid relations, (br1) and (br2), as the generators of $\mathfrak{S}_n$, Matsumoto's lemma implies that the element 
$g_w:=g_{i_1}g_{i_2}\ldots g_{i_r}$
is well-defined, that is, it does not depend on the choice of the reduced expression for $w$.
We then obtain an $\cR$-basis of $\Y$ as follows:

\begin{thm}\label{theo-bases}\cite[Theorem 4.15]{ChPo2}
The set 
$$\mathcal{B}^{\mathrm{aff}}_{d,n}= \left\{t_1^{a_1}\cdots t_n^{a_n} \,X_1^{b_1}\cdots X_n^{b_n}\, g_w\,\left|\,\, a_1,\ldots,a_n\in \Z/d\Z,\,\,b_1,\ldots,b_n \in \Z,\,\, w \in \mathfrak{S}_n \right.\right\}$$
 is an $\cR$-basis of $\Y$. 
\end{thm}

For $d=1$, $\mathcal{B}^{\mathrm{aff}}_{1,n}$ is the standard Bernstein basis of the affine Hecke algebra  ${\rm Y}_{1,n}^{\mathrm{aff}}(q)$ of type $A$.

\section{The pro-$p$-Iwahori--Hecke algebra}

Let $\Lambda$ denote the abelian group $\Z^n$. 
For $\lambda = (\lambda_1,\ldots,\lambda_n), \, \lambda' = (\lambda_1',\ldots,\lambda_n') \in \Lambda$, we have 
$\lambda \lambda ' = \lambda'\lambda = (\lambda_1+\lambda_1',\ldots,\lambda_n+\lambda_n')$. We denote by
$\lambda \circ \lambda'$ the dot product of $\lambda$ and $\lambda'$, that is, the integer $\sum_{j=1}^n\lambda_j\lambda_j'$.
We set $\varepsilon_i : = (0,0,\ldots,0,1,-1,0,\ldots,0) \in \Lambda$, where $1$ is in the $i$-th position and $-1$ is in the $(i+1)$-th position, for $i=1,\ldots,n-1$. 

Let $W$ be the extended affine Weyl group $\Lambda \rtimes \mathfrak{S}_n$ of type $A$. 
For all $\sigma \in \mathfrak{S}_n$ and $\lambda=(\lambda_1,\ldots,\lambda_n) \in \Lambda$, we set
$$ \sigma(\lambda) := \sigma \lambda \sigma^{-1} = (\lambda_{\sigma(1)},\ldots,\lambda_{\sigma(n)}). $$

Now, the symmetric group $ \mathfrak{S}_n $ is generated by the set 
$S =\{s_1,\ldots,s_{n-1}\}$, where $s_i$ denotes the transposition $(i,i+1)$. 
We set
$$
\gamma:= s_{n-1}s_{n-2} \ldots s_2s_1 \in \mathfrak{S}_n \quad \text{and} \quad  h:=(0,0,\ldots,0,1)\, \gamma \in W.
$$
Note that we have $hs_ih^{-1}=s_{i-1}$ for all $i=2,\dots,n-1$. Set $s_0 := hs_1h^{-1} \in W$.
Then the set $S^{\mathrm{aff}} =S \cup \{s_0\}$ is a generating set of the 
affine Weyl group $W^{\mathrm{aff}}$ of type $A$ and we have 
$W = \langle h \rangle \rtimes W^{\mathrm{aff}}$. Moreover, we can extend the 
length function $\ell$ of $W^{\mathrm{aff}}$ to $W$ by setting $\ell(h^k 
w^{\mathrm{aff}}) = \ell(w^{\mathrm{aff}})$ for all $w^{\mathrm{aff}} \in 
W^{\mathrm{aff}}$, $k \in \Z$.

Let $\cT$ be a (finite) abelian group such that $\mathfrak{S}_n$ acts on $\cT$. Like above, for all $\sigma \in \mathfrak{S}_n$ and $t\in \cT$, we set
$\sigma(t):=\sigma t \sigma^{-1}$.

We consider the semi-direct product $\widetilde{W} := (\cT \times \Lambda) \rtimes \mathfrak{S}_n$.
Every element of $\widetilde{W}$ can be written in the form $t \, \lambda \, \sigma$, with $t \in \cT$, $\lambda \in \Lambda$, $\sigma \in \mathfrak{S}_n$. Note that $t$ and $\lambda$ commute with each other. We set  $\widetilde{\Lambda}:= \cT \times \Lambda$ and $\widetilde{\mathfrak{S}}_n:= \cT \rtimes \mathfrak{S}_n$. We have $ \widetilde{W} = \widetilde{\Lambda}\widetilde{\mathfrak{S}}_n$, with
$ \widetilde{\Lambda} \cap \widetilde{\mathfrak{S}}_n = \cT$. Like above, for all $\sigma \in \mathfrak{S}_n$ and $\nu \in \widetilde{\Lambda}$, we set 
$\sigma(\nu):=\sigma \nu \sigma^{-1}$. We also set $\widetilde{S} : = \{ ts \,| \, s \in S,\,t \in \cT \}$ and $\widetilde{S}^{\mathrm{aff}}: = \{ ts \,| \, s \in S^{\mathrm{aff}},\,t \in \cT \}$.
Finally, we can extend the length function $\ell$ of $W$  to $\widetilde{W}$ by setting $\ell(tw) = \ell(w)$ for all $w \in W$, $t\in \cT$.

\begin{thm}\label{thm1}{\rm \cite[Theorem 2.4]{Vi2}}
Let $R$ be a ring and let $(q_s, c_s)_{s \in \widetilde{S}^{{\rm aff}}} \in R \times R[\cT]$ be such that 
\begin{enumerate}[(a)]
\item $q_s = q_{ts}$ and $c_{ts} = t c_s $  for all $t \in \cT$. \smallbreak
\item $q_s = q_{s'}$ and $c_{s'} = wc_sw^{-1}$ if $s' = wsw^{-1}$ for some $w \in \widetilde{W}$. \smallbreak
\end{enumerate}
Then the free $R$-module $\cH_R(q_s, c_s)$ of basis $(T_w)_{w \in \widetilde{W}}$ has a unique $R$-algebra structure satisfying: 
\begin{itemize}
\item The braid relations: $T_wT_{w'} =T_{ww'}$ if $w,w' \in \widetilde{W}$, $\ell(w)+\ell(w')=\ell(ww')$. \smallbreak
\item The quadratic relations: $T_s^2 = q_ss^2 + c_sT_s$ for $s \in \widetilde{S}^{{\rm aff}}$. \smallbreak
\end{itemize}
\end{thm}

The algebra $\cH_R(q_s, c_s)$ is called the \emph{pro-$p$-Iwahori--Hecke 
  algebra of} $\widetilde{W}$ over $R$. If $q_s$ is  invertible in $R$  for all $s \in \widetilde{S}^{{\rm aff}}$, then  $T_s$ is invertible in $\cH_{R}(q_s, c_s)$, with
\begin{equation}\label{E-inv}
T_s^{-1} = q_{s}^{-1}s^{-2} \left(T_s - c_{s} \right).
\end{equation}
We deduce that every element $T_w$, for $w \in \widetilde{W}$, is invertible in $\cH_{R}(q_s, c_s)$. 
If, further, $q_s^{1/2} \in R$ for all  $s \in \widetilde{S}^{{\rm aff}}$, then, by replacing the generators $T_s$ by $\bar{T}_s:=q_s^{-1/2}T_s$, the quadratic relations become
\begin{center}
$\bar{T}_s^2 = s^2 + q_s^{-1/2}c_s\bar{T}_s$ \,\,for all\, $s \in \widetilde{S}^{{\rm aff}}$.
\end{center}
Thus, if  $(q_s^{1/2})_{s \in\widetilde{S}^{{\rm aff}}}$ are chosen so that they also satisfy conditions $(a)$ and $(b)$ of Theorem \ref{thm1}, we obtain an isomorphism
between $\cH_{R}(q_s, c_s)$  and $\cH_{R}(1, q_s^{-1/2}c_s)$. 
Therefore, without loss of generality, we may assume that $q_s=1$ for all $s \in  \widetilde{S}^{{\rm aff}}$.  \smallbreak

We now have the following Bernstein presentation of  $\cH_R(1, c_s)$.

\begin{thm}\label{thm2}{\rm \cite[Theorem 2.10]{Vi2}}
The $R$-algebra  $\cH_{R}(1, c_s)$  is isomorphic to the free $R$-module of basis
$(E(w))_{w \in \widetilde{W}}$ endowed with the unique $R$-algebra structure satisfying: 
\begin{itemize}
\item Braid relations: $E(w)E(w') = E(ww')$ for $w, w' \in \widetilde{\mathfrak{S}}_n$, $\ell(w) + \ell(w') = \ell(ww')$. \smallbreak
\item Quadratic relations: $E(s)^2 = s^2 + c_sE(s)$ for $s \in \widetilde{S}$.  \smallbreak
\item Product relations: $E(\nu)E(w) = E(\nu w)$ for $\nu \in \widetilde{\Lambda}$, $w \in \widetilde{W}$.  \smallbreak
\item Bernstein relations: For $s=ts_i \in \widetilde{S}$ ($t \in \cT$, $i=1,\ldots,n-1$) and $\nu = \tau \lambda \in \widetilde{\Lambda}$ ($\tau \in \cT$, $\lambda \in \Lambda$), 
$$E(s(\nu))E(s) - E(s)E(\nu) =
\left\{\begin{array}{ll}
0 & \text{if } s_i(\lambda)=\lambda,\\ [0.2em]
  \epsilon_i (\lambda) \,c_{s} \sum_{k=0}^{|\varepsilon_i \circ \lambda|-1} E(\tau\mu_{i}(k,\lambda)) 
&\text{if } s_i(\lambda) \neq \lambda,\\
\end{array}\right.$$ 
\end{itemize} where 
$ \epsilon_i (\lambda) = {\rm sign}(\varepsilon_i \circ \lambda) \in \{1,-1\}$ and
$$
\mu_{i}(k,\lambda) = \left\{\begin{array}{ll}
\lambda \,\varepsilon_i^k & \text{if }  \epsilon_i (\lambda) =-1\\ [0.2em]
s_i(\lambda) \,\varepsilon_i^k  & \text{if }  \epsilon_i (\lambda) =1\,.\\
\end{array}\right.
$$
\end{thm}

\begin{rem}\label{explain inv}{\rm
Note that, in the above presentation, we take $E(s) = T_s$ for all $s \in \widetilde{S}$ (in particular, $E(t)=t$ for all $t \in \cT$).}
\end{rem}

\section{Main result}

Let $q$ be an indeterminate or a non-zero complex number. 
Our aim in this section will be to show that the affine Yokonuma--Hecke algebra $\Y$ is isomorphic to the  pro-$p$-Iwahori--Hecke
 algebra  $\cH_{R}(q_s, c_s) \text{ of } \widetilde{W}$  when we take 
\begin{itemize}
\item $\cT = (\Z /d \Z)^n = \langle \, t_1,\ldots,t_n \,\,|\,\, t_j^d=1,\,\,t_it_j=t_jt_i, \,\,\text{for all } i,j=1,\ldots,n \,\rangle$ ; \smallbreak
\item $R = \cR = \C[q,q^{-1}]$ ; \smallbreak
\item $q_{ts_i}=1$ for all $i=0,1,\ldots,n-1$, $t \in  (\Z /d \Z)^n$ ; \smallbreak
\item $c_{ts_i} = (q-q^{-1})\, t\,e_i$  for all $i=0,1,\ldots,n-1$, $t \in  (\Z /d \Z)^n$, where $e_0:=e_{1,n}$.
\end{itemize}
We then have $\widetilde{W} = ( \Z /d \Z \times \Z)^n \rtimes \mathfrak{S}_n =  \{ t \,\lambda\, \sigma\,|\, t \in  (\Z /d \Z)^n,\, \lambda \in \Lambda,\,\sigma \in \mathfrak{S}_n\}$. 
The action of $\mathfrak{S}_n$ on $ (\Z /d \Z)^n $ is given by
$$ \sigma(t_j) = \sigma t_j \sigma^{-1} = t_{\sigma(j)} \quad \text{ for all } \sigma \in \mathfrak{S}_n, \, j=1,\ldots,n.$$
The action of $\mathfrak{S}_n$ extends linearly to the group algebra ${\cR}[(\Z /d \Z)^n]$.

First we check that the assumptions $(a)$ and $(b)$ of Theorem \ref{thm1} are satisfied in this case.
Let $s \in \widetilde{S}^{\rm aff} = \{ts_i \,|\, i=0,1,\ldots,n-1, t \in  (\Z /d \Z)^n\}$.  
By definition, we have $q_{s} = q_{ts} = 1$ and $c_{ts} =t c_{s}$ for all $t \in  (\Z /d \Z)^n$.

Now, set
$$\lambda_0:=(-1,0,\ldots,0,1) \in \Lambda \quad \text{ and } \quad \lambda_i:= (0,0,\ldots,0,0) \in \Lambda \quad \text{for all } i=1,\ldots,n-1.$$
Moreover, let $\sigma_0$ denote the transposition $(1,n)=s_{1,n}$ and $\sigma_i$ denote the transposition $(i,i+1)=s_i$ for all $i=1,\ldots,n-1$.
We then have
$$s_i = \lambda_i \sigma_i \quad \text{for all } i=0,1,\ldots,n-1.$$
So
$$\widetilde{S}^{\rm aff} = \{t\lambda_i \sigma_i \,|\, i=0,1,\ldots,n-1, t \in  (\Z /d \Z)^n\}.$$

Let $s,s' \in \widetilde{S}^{\rm aff} $ be conjugate in $\widetilde{W}$. 
By definition, we have $q_{s} =  q_{s'}=1$.
Now, let us write $s=t\lambda_i\sigma_i$ and $s' = t'\lambda_{i'}\sigma_{i'}$ for $i,i' \in \{0,1,\ldots,n-1\}$ and $t,t' \in  (\Z /d \Z)^n$.
Moreover, let $w \in \widetilde{W}$ be such that  $s' = wsw^{-1}$, and write $w = \tau \lambda \sigma$
with $\tau \in (\Z /d \Z)^n$, $\lambda \in \Lambda$  and $\sigma \in \mathfrak{S}_n$.
Then
$$t'\lambda_{i'}\sigma_{i'}= ( \tau \lambda \sigma) (t \lambda_i\sigma_i) (\sigma^{-1} \lambda^{-1} \tau^{-1}) =  [ \tau \sigma(t) (\sigma \sigma_i \sigma^{-1}) (\tau^{-1})] [\lambda \sigma(\lambda_i) (\sigma \sigma_i \sigma^{-1})  (\lambda^{-1})]
[\sigma \sigma_i \sigma^{-1}].$$
We deduce that 
$$\sigma_{i'} = \sigma \sigma_i \sigma^{-1} \quad\text{ and }\quad t' =  \tau \sigma(t) (\sigma \sigma_i \sigma^{-1}) (\tau^{-1})= \tau \sigma(t)  \sigma_{i'}(\tau^{-1}).$$
By  \eqref{useful e_i}, we have
\begin{equation}\label{cond1}
 t' e_{i'} =  \tau \sigma(t)  \sigma_{i'}(\tau^{-1}) e_{i'} =  \tau \sigma(t) \tau^{-1} e_{i'} = \sigma(t) e_{i'} = \sigma(t) w w^{-1} e_{i'} = wtw^{-1}e_{i'}  .
 \end{equation}
Furthermore, we have
$$\sigma_{i'} = \sigma \sigma_i \sigma^{-1} = s_{\sigma(i), \sigma(i+1)}=(\sigma(i), \sigma(i+1)),$$ 
which implies that
\begin{equation}\label{cond2}
we_i w^{-1} = \sigma(e_i) w w^{-1} = \sigma(e_i) = e_{\sigma(i), \sigma(i+1)} = e_{i'}. 
\end{equation}
Combining \eqref{cond1} and \eqref{cond2}, we obtain
$$c_{s'}=(q-q^{-1})t'e_{i'}=(q-q^{-1})  wtw^{-1} we_i w^{-1}  = w \left( (q-q^{-1}) t e_i \right) w^{-1} = wc_{s}w^{-1}.$$

Therefore, we can define the pro-$p$-Iwahori--Hecke algebra $\cH_{{\cR}}(1, c_s)$ of $\widetilde{W}$, which is the free ${\cR}$-module of basis
$(E(w))_{w \in \widetilde{W}}$  endowed with the unique ${\cR}$-algebra structure satisfying the relations described explicitly in the previous section.

\begin{thm}\label{main}
The ${\cR}$-linear map $\varphi : \Y \rightarrow \cH_{{\cR}}(1, c_s)$ defined by
\begin{equation}\label{bas to bas}
\varphi(t_1^{a_1}\cdots t_n^{a_n} \,X_1^{b_1}\cdots X_n^{b_n}\, g_w) = E( t_1^{a_1}\cdots t_n^{a_n}\, (b_1,\ldots,b_n)\, w)\ ,
\end{equation}
for all $a_1,\ldots,a_n\in \Z/d\Z$, $b_1,\ldots,b_n \in \Z,$ and $w \in \mathfrak{S}_n$,
is an ${\cR}$-algebra isomorphism.
\end{thm}

\begin{proof}
We will show first that we can define an ${\cR}$-algebra homomorphism $\varphi : \Y \rightarrow \cH_{{\cR}}(1, c_s)$ given by
$$\begin{array}{rcll}
\varphi(t_j)&=&E(t_j) &\text{ for all } j=1,\ldots,n, \\ [0.2em]
\varphi(g_i)&=&E(s_i) & \text{ for all } i=1,\ldots,n-1,\\  [0.2em]
\varphi(X_1)&=&E((1,0,\ldots,0)), & \\[0.2em]
 \varphi(X_1^{-1})&=&E((-1,0,\ldots,0)). &\\[0.2em]
 \end{array}$$
For this, it is enough to check that the defining relations (\ref{def-aff1}) of $\Y$ are satisfied by the images of its generators via $\varphi$, that is, the elements
$$\varphi(t_1),\ldots,\varphi(t_n),\varphi(g_1),\ldots,\varphi(g_{n-1}),\varphi(X_1),\varphi(X_1^{-1}) .$$ 

First of all, note that $\varphi(X_1^{-1})=E((-1,0,\ldots,0)) = E((1,0,\ldots,0))^{-1} = \varphi(X_1)^{-1}$, due to the product relations. The product relations also imply (aff3).
Moreover, due to the braid relations, we have immediately that (br1), (br2), (fr1) and (fr3) are satisfied. 
We also obtain 
$$\varphi(g_i)\varphi(t_j)=E(s_i)E(t_j) = E(s_it_j) = E(t_{s_i(j)}s_i)=E(t_{s_i(j)})E(s_i)=\varphi(t_{s_i(j)})\varphi(g_i)$$
for all $i=1,\ldots,n-1$ and $j=1,\ldots,n$, so (fr2) holds.

Now, for $i=1,\ldots,n-1$, we have
$$\varphi(g_i)^2= E(s_i)^2 = s_i^2 + c_{s_i}E(s_i) = 1 + (q-q^{-1})e_iE(s_i) = 1 + (q-q^{-1})e_i \varphi(g_i),$$
so (quad) is satisfied by $\varphi(g_i)$.
We will use the Bernstein relations for the remaining defining relations, (aff1) and (aff2).

First, note that $(1,0,\ldots,0)$ is fixed by the action of $s_i$ for all $i=2,\ldots,n-1$. 
We thus obtain
$$\varphi(X_1)\varphi(g_i) = E((1,0,\ldots,0))E(s_i) = E(s_i(1,0,\ldots,0)s_i^{-1})E(s_i) = E(s_i) E((1,0,\ldots,0))=\varphi(g_i) \varphi(X_1).$$
This yields (aff2).

Now, $(1,0,\ldots,0)$ is not fixed by the action of $s_1$.  
By the Bernstein relations, we have
$$E((0,1,\ldots,0))E(s_1) - E(s_1)E((1,0,\ldots,0)) = c_{s_1}E((0,1,\ldots,0)),$$
and thus,
$$E(s_1)E((1,0,\ldots,0)) = E((0,1,\ldots,0))E(s_1) - c_{s_1} E((0,1,\ldots,0)) =  E((0,1,\ldots,0)) (E(s_1) - c_{s_1}).$$
By (\ref{E-inv}) and Remark \ref{explain inv}, we have that $E(s_1) - c_{s_1} = E(s_1)^{-1}$, and so
\begin{equation}\label{similar}
E(s_1)E((1,0,\ldots,0)) =E((0,1,\ldots,0)) E(s_1)^{-1}.
\end{equation}
We  obtain
$$\begin{array}{rcl}
\varphi(X_1)\varphi(g_1)\varphi(X_1)\varphi(g_1) & = & E((1,0,\ldots,0))E(s_1)E((1,0,\ldots,0))E(s_1) \\[0.2em]
 &  = &  E((1,0,\ldots,0))E((0,1,\ldots,0)) E(s_1)^{-1}E(s_1) \\[0.2em]
 & =   &E((1,0,\ldots,0))E((0,1,\ldots,0)).\\[0.2em]
 & =   &E((1,1,\ldots,0)).\\[0.2em]
 \end{array}$$
and
$$\begin{array}{rcl}
\varphi(g_1)\varphi(X_1)\varphi(g_1)\varphi(X_1) & = & E(s_1)E((1,0,\ldots,0))E(s_1)E((1,0,\ldots,0)) \\[0.2em]
 &  = &  E((0,1,\ldots,0))E(s_1)^{-1} E(s_1) E((1,0,\ldots,0)) \\[0.2em]
 & =   &E((0,1,\ldots,0))E((1,0,\ldots,0)).\\[0.2em]
 & =   &E((1,1,\ldots,0)).\\[0.2em]
 \end{array}$$
 Thus, (aff1) also holds, and $\varphi$ is an $\cR$-algebra homomorphism.
 
 We will now prove that $\varphi$ is indeed the $\cR$-linear map given by \eqref{bas to bas}. First, for all $i=1,\ldots,n-1$, we have
 $$\begin{array}{rcl}
 \varphi(X_{i+1})&=& \varphi(g_ig_{i-1}\ldots g_1X_1 g_1\ldots g_{i-1}g_i) \\[0.2em]
 &=& \varphi(g_i) \varphi(g_{i-1})\ldots  \varphi(g_1) \varphi(X_1)  \varphi(g_1)\ldots  \varphi(g_{i-1}) \varphi(g_i) \\[0.2em]
 &=&E(s_i) E(s_{i-1})\ldots  E(s_1)E((1,0,\ldots,0))  E(s_1)\ldots E(s_{i-1}) E(s_i). \\[0.2em]
 \end{array}$$
 Using repeatedly the Bernstein relations (similarly to \eqref{similar}), we obtain that
 $$\varphi(X_{i+1})=E((0,0,\ldots,0,1,0,\ldots,0)), $$
where $1$ is in position $i+1$. Now, by the product and braid relations, we obviously get
$$\varphi(t_1^{a_1}\cdots t_n^{a_n} \,X_1^{b_1}\cdots X_n^{b_n}\, g_w) = E( t_1^{a_1}\cdots t_n^{a_n}\, (b_1,\ldots,b_n)\, w),$$
for all $a_1,\ldots,a_n\in \Z/d\Z$, $b_1,\ldots,b_n \in \Z,$ and $w \in \mathfrak{S}_n$.
 
 Finally, by Theorem \ref{theo-bases} and Theorem \ref{thm2}, the map $\varphi$ is bijective, so $\varphi$ is 
 an $\cR$-algebra isomorphism, as required.
\end{proof}

Now, if $q^2$ is a power of a prime number $p$ and $d=q^2-1$, then  the pro-$p$-Iwahori--Hecke algebra $\cH_{{\cR}}(1, c_s)$
defined above is the convolution algebra 
of complex valued and compactly supported functions on the group 
${\rm GL}_{n}(F)$, with $F$ a suitable $p$-adic field, that are bi-invariant 
under the pro-$p$-radical of an Iwahori subgroup \cite[Theorem 1]{Vi1}. Therefore, the following result, also stated in
\cite[Remark 2.2]{ChPo2},
is an immediate consequence of Theorem \ref{main}.

%

\begin{cor}
If $q^2=p^k$, where $p$ is a prime number and $k \in \mathbb{Z}_{>0}$, and $d=q^2-1$,
then the affine Yokonuma--Hecke 
algebra $\Y$ is isomorphic to the convolution algebra 
of complex valued and compactly supported functions on the group 
${\rm GL}_{n}(F)$, where
 $F$ is a local non-archimedean field of residual field $\mathbb{F}_{q^2}$, that are bi-invariant 
under the pro-$p$-radical of an Iwahori subgroup.
\end{cor}

\end{document}